\documentclass[reqno,11pt]{article} 
\usepackage{amsmath, amssymb}
\usepackage{amsthm} 
\theoremstyle{plain} 
\newtheorem{theorem}{\indent\sc Theorem}[section]
\newtheorem{lemma}[theorem]{\indent\sc Lemma}

\newtheorem{proposition}[theorem]{\indent\sc Proposition}

\theoremstyle{definition} 

\newtheorem{remark}[theorem]{\indent\sc Remark}

\makeatletter
\def\address#1#2{\begingroup
\noindent\parbox[t]{7.8cm}{%
\small{\scshape\ignorespaces#1}\par\vskip1ex
\noindent\small{\itshape E-mail address}%
\/: #2\par\vskip4ex}\hfill%
\endgroup}%
\makeatother
\title{Application of Weierstrass units to relative power integral bases} 
\author{
\textsc{Ho Yun Jung, Ja Kyung Koo and Dong Hwa Shin$^*$} 
}
\date{} 
\begin{document}

\maketitle

\footnote{ 
2010 \textit{Mathematics Subject Classification}. Primary 11G16, Secondary 11G15, 11Y40.}
\footnote{ 
\textit{Key words and phrases}. Power integral bases, Shimura's reciprocity law,
Weierstrass units.}
\footnote{
\thanks{
The second named author was partially supported by the NRF of Korea
grant funded by MEST (2012-0000798). $^*$The corresponding author was
supported by Hankuk University of Foreign Studies Research Fund of
2012.} }

\begin{abstract}
Let $K$ be an imaginary quadratic field other than $\mathbb{Q}(\sqrt{-1})$ and
$\mathbb{Q}(\sqrt{-3})$. We construct relative power integral bases between certain abelian extensions
of $K$ in terms of Weierstrass units.
\end{abstract}

\section {Introduction}

Let $L/F$ be an extension of number fields and let $\mathcal{O}_L$ and $\mathcal{O}_F$
be the rings of integers
of $L$ and $F$, respectively.
We say that an element $\alpha$ of $L$ forms a \textit{relative power integral basis} for
$L/K$ if $\mathcal{O}_L=\mathcal{O}_F[\alpha]$. For example, if $N$ is a positive integer, then
$\zeta_N=e^{2\pi i/N}$ forms a (relative)  power integral basis for the extension
$\mathbb{Q}(\zeta_N)/\mathbb{Q}$ \cite[Theorem 2.6]{Washington}. And, in general not much has been known
about relative power integral bases except for extensions of degree less than or equal to $9$ (\cite{Gaal1}--\cite{Gaal12}).
\par
Let $K$ be an imaginary quadratic field other than $\mathbb{Q}(\sqrt{-1})$ and $\mathbb{Q}(\sqrt{-3})$.
Let $m$ and $n$ be positive integers
such that $m$ has at least two prime factors and each prime factor of $mn$ splits in $K/\mathbb{Q}$.
In this paper we shall
show that certain Weierstrass unit forms a relative power integral basis for
the ray class field modulo $(mn)$ over the compositum of the ray class field modulo $(m)$ and the ring class field of the order of conductor $mn$ of $K$
(Theorem \ref{main}).
To this end, we shall make use of an explicit description of Shimura's reciprocity law due to Stevenhagen \cite{Stevenhagen}.

\section {Weierstrass units}

For a positive integer $N$, let $\Gamma(N)$ be the principal congruence subgroup
of level $N$, namely
\begin{equation*}
\Gamma(N)=\{\gamma\in\mathrm{SL}_2(\mathbb{Z})~|~\gamma\equiv I_2\pmod{N}\}.
\end{equation*}
Then $\overline{\Gamma}(N)=
\Gamma(N)/\{\pm I_2\}$ acts on the complex upper half-plane $\mathbb{H}$ by fractional linear transformations.
Let $\mathcal{F}_N$ be the field of meromorphic modular functions for $\overline{\Gamma}(N)$ (or, of level $N$)
 whose Fourier coefficients lie in the $N$th cyclotomic field $\mathbb{Q}(\zeta_N)$.
As is well-known,
$\mathcal{F}_1$ is generated by the elliptic modular function
\begin{equation*}
j(\tau)=q^{-1}+744+196884q+21493760q^2
+864299970q^3+\cdots\quad(q=e^{2\pi i\tau})
\end{equation*}
over $\mathbb{Q}$ \cite[Chapter 6]{Lang}.
Furthermore, $\mathcal{F}_N$ is a Galois extension of $\mathcal{F}_1$
whose Galois group is isomorphic to
$\mathrm{GL}_2(\mathbb{Z}/N\mathbb{Z})/\{\pm I_2\}$ \cite[Chapter 6, Theorem 3]{Lang}. Let
$\mathcal{R}_N$ and $\mathbb{Q}\mathcal{R}_N$ be the integral closures
of $\mathbb{Z}[j(\tau)]$ and $\mathbb{Q}[j(\tau)]$ in
$\mathcal{F}_N$, respectively.
We call the elements of
$(\mathbb{Q}\mathcal{R}_N)^*$ \textit{modular units} of level $N$,
which are precisely those
elements of $\mathcal{F}_N$ having no zeros and poles on
$\mathbb{H}$ \cite[p.36]{K-L}. In particular, we call the elements of $\mathcal{R}_N^*$ \textit{modular units over $\mathbb{Z}$} of level $N$.
\par
Let $\Lambda=[\omega_1,\omega_2]$ ($=\omega_1\mathbb{Z}+\omega_2\mathbb{Z}$) be a lattice in $\mathbb{C}$. The \textit{Weierstrass $\wp$-function} relative to $\Lambda$ is defined by
\begin{equation*}
\wp(z;\Lambda)=\frac{1}{z^2}+\sum_{\omega\in\Lambda-\{0\}}
\bigg\{\frac{1}{(z-\omega)^2}-\frac{1}{\omega^2}\bigg\}\quad(z\in\mathbb{C}),
\end{equation*}
which is a meromorphic function on $z$ and periodic with respect to $\Lambda$.

\begin{lemma}\label{pm}
Let $z_1,z_2\in\mathbb{C}-\Lambda$. Then,
$\wp(z_1;\Lambda)=\wp(z_2;\Lambda)$ if and only if $z_1\equiv\pm z_2\pmod{\Lambda}$.
\end{lemma}
\begin{proof}
See \cite[Chaper IV, $\S$3]{Silverman}.
\end{proof}

Let $\left[\begin{matrix}r\\s\end{matrix}\right]\in(1/N)\mathbb{Z}^2-\mathbb{Z}^2$
for an integer $N$ ($\geq2$). We define
\begin{equation*}
\wp_{\left[\begin{smallmatrix}r\\s\end{smallmatrix}\right]}(\tau)=\wp(r\tau+s;[\tau,1])\quad(\tau\in\mathbb{H}).
\end{equation*}
It is a weakly holomorphic (that is, holomorphic on $\mathbb{H}$) modular form of
level $N$ and weight $2$ \cite[Chapter 6]{Lang}. We further define
\begin{equation*}
g_2(\tau)=60\sum_{\omega\in[\tau,1]-\{0\}}\frac{1}{\omega^4},\quad
g_3(\tau)=140\sum_{\omega\in[\tau,1]-\{0\}}\frac{1}{\omega^6},\quad
\Delta(\tau)=g_2(\tau)^3-27g_3(\tau)^2,
\end{equation*}
which are modular forms of level $1$ and weight $4$, $6$ and $12$, respectively. Now we define the \textit{Fricke function} $f_{\left[\begin{smallmatrix}r\\s\end{smallmatrix}\right]}(\tau)$  by
\begin{equation}\label{Fricke}
  f_{\left[\begin{smallmatrix}r\\s\end{smallmatrix}\right]}(\tau)=\frac{g_2(\tau)g_3(\tau)}{\Delta(\tau)}
  \wp_{\left[\begin{smallmatrix}r\\s\end{smallmatrix}\right]}(\tau).
\end{equation}
It depends only on
$\pm\left[\begin{matrix}r\\s\end{matrix}\right]\pmod{\mathbb{Z}^2}$ \cite[p.8]{Lang} and
is weakly holomorphic because
$\Delta(\tau)$ does not vanish on $\mathbb{H}$.

\begin{lemma}\label{Fricketransf}
$f_{\left[\begin{smallmatrix}r\\s\end{smallmatrix}\right]}(\tau)$ belongs to $\mathcal{F}_N$ and satisfies the transformation formula
\begin{equation*}
f_{\left[\begin{smallmatrix}r\\s\end{smallmatrix}\right]}(\tau)^\gamma=
f_{{}^t\gamma\left[\begin{smallmatrix}r\\s\end{smallmatrix}\right]}(\tau)\quad
(\gamma\in\mathrm{GL}_2(\mathbb{Z})/\{\pm I_2\}\simeq\mathrm{Gal}(\mathcal{F}_N/\mathcal{F}_1)),
\end{equation*}
where ${}^t\gamma$ stands for the transpose of $\gamma$.
\end{lemma}
\begin{proof}
See \cite[Chapter 6, $\S$2--3]{Lang}.
\end{proof}

On the other hand, we define the
\textit{Siegel function} $g_{\left[\begin{smallmatrix}r\\s\end{smallmatrix}\right]}(\tau)$ by
\begin{equation*}
g_{\left[\begin{smallmatrix}r\\s\end{smallmatrix}\right]}(\tau)=
-q^{(1/2)
(r^2-r+1/6)}e^{\pi is(r-1)}(1-q^re^{2\pi
is})\prod_{n=1}^{\infty}(1-q^{n+r}e^{2\pi is})(1-q^{n-r}e^{-2\pi
is}).
\end{equation*}

\begin{lemma}\label{Siegellemma}
Let $M$ be the primitive denominator of
$\left[\begin{matrix}
r\\s\end{matrix}\right]$ \textup{(}that is, $M$ is the least positive integer so that
$Mr,Ms\in\mathbb{Z}$\textup{)}.
\begin{itemize}
\item[\textup{(i)}] $g_{\left[\begin{smallmatrix}
r\\s\end{smallmatrix}\right]}(\tau)^{12M}$
and $g_{\left[\begin{smallmatrix}
r\\s\end{smallmatrix}\right]}(\tau)$ are modular units of level $M$ and $12M^2$, respectively.
\item[\textup{(ii)}] $g_{\left[\begin{smallmatrix}
r\\s\end{smallmatrix}\right]}(\tau)^{12M}$ depends only on $\pm\left[\begin{matrix}r\\s\end{matrix}\right]\pmod{\mathbb{Z}^2}$
and satisfies the transformation formula
\begin{equation*}
(g_{\left[\begin{smallmatrix}
r\\s\end{smallmatrix}\right]}(\tau)^{12M})^\gamma=
g_{{}^t\gamma\left[\begin{smallmatrix}
r\\s\end{smallmatrix}\right]}(\tau)^{12M}\quad(\gamma\in\mathrm{GL}_2(\mathbb{Z}/M\mathbb{Z})/\{\pm I_2\}\simeq
\mathrm{Gal}(\mathcal{F}_M/\mathcal{F}_1)).
\end{equation*}
\item[\textup{(iii)}] Moreover, if $M$ has at least two prime factors, then $g_{\left[\begin{smallmatrix}
r\\s\end{smallmatrix}\right]}(\tau)$ is a modular unit over $\mathbb{Z}$.
\end{itemize}
\end{lemma}
\begin{proof}
\begin{itemize}
\item[(i)] See \cite[Chapter 2, Theorem 1.2 and Chapter 3, Theorems 5.2 and 5.3]{K-L}.
\item[(ii)] See \cite[Chapter 2, Proposition 1.4]{K-L}.
\item[(iii)] See \cite[Chapter 2, Theorem 2.2(i)]{K-L}.
\end{itemize}
\end{proof}

\begin{lemma}\label{wtog} Let
$\left[\begin{matrix}a\\b\end{matrix}\right],
\left[\begin{matrix}c\\d\end{matrix}\right]\in\mathbb{Q}^2-\mathbb{Z}^2$
such that $\left[\begin{matrix}a\\b\end{matrix}\right]\not\equiv\pm\left[\begin{matrix}c\\d\end{matrix}\right]\pmod{\mathbb{Z}^2}$.
We have the relation
\begin{equation*}
  \wp_{\left[\begin{smallmatrix}a\\b\end{smallmatrix}\right]}(\tau)-
  \wp_{\left[\begin{smallmatrix}c\\d\end{smallmatrix}\right]}(\tau)=
  -\frac{g_{\left[\begin{smallmatrix}a+c\\b+d\end{smallmatrix}\right]}(\tau)
  g_{\left[\begin{smallmatrix}a-c\\b-d\end{smallmatrix}\right]}(\tau)\eta(\tau)^4}
  {g_{\left[\begin{smallmatrix}a\\b\end{smallmatrix}\right]}(\tau)^2
  g_{\left[\begin{smallmatrix}c\\d\end{smallmatrix}\right]}(\tau)^2},
\end{equation*}
where
\begin{equation*}
  \eta(\tau)=\sqrt{2\pi}\zeta_8q^{1/24}\prod_{n=1}^\infty(1-q^n).
\end{equation*}
\end{lemma}
\begin{proof}
See \cite[p.51]{K-L}.
\end{proof}

Let $m$ ($\geq2$) and $n$ be positive integers. We define
\begin{equation}
h_{m,n}(\tau)=\frac{\wp_{\left[\begin{smallmatrix}0\\1/mn\end{smallmatrix}\right]}(\tau)
-\wp_{\left[\begin{smallmatrix}1/m\\0\end{smallmatrix}\right]}(\tau)}
{\wp_{\left[\begin{smallmatrix}0\\1/m\end{smallmatrix}\right]}(\tau)
-\wp_{\left[\begin{smallmatrix}1/m\\0\end{smallmatrix}\right]}(\tau)}.
\end{equation}

\begin{proposition}\label{h}
$h_{m,n}(\tau)$ is a modular unit of level $mn$. And,
if $m$ has at least two prime factors, then $h_{m,n}(\tau)$ is a modular unit over $\mathbb{Z}$.
\end{proposition}
\begin{proof}
It follows rom Lemma \ref{pm} that the denominator of $h_{m,n}(\tau)$ is not the zero function.
Furthermore, since
\begin{equation}\label{hf}
h_{m,n}(\tau)=\frac{f_{\left[\begin{smallmatrix}0\\1/mn\end{smallmatrix}\right]}(\tau)
-f_{\left[\begin{smallmatrix}1/m\\0\end{smallmatrix}\right]}(\tau)}
{f_{\left[\begin{smallmatrix}0\\1/m\end{smallmatrix}\right]}(\tau)
-f_{\left[\begin{smallmatrix}1/m\\0\end{smallmatrix}\right]}(\tau)}
\end{equation}
by the definition (\ref{Fricke}), it belongs to $\mathcal{F}_{mn}$ by Lemma \ref{Fricketransf}.
\par
On the other hand, we see that
\begin{eqnarray*}
h_{m,n}(\tau)&=&\frac{-g_{\left[\begin{smallmatrix}1/m\\1/mn\end{smallmatrix}\right]}(\tau)
g_{\left[\begin{smallmatrix}-1/m\\1/mn\end{smallmatrix}\right]}(\tau)\eta(\tau)^4/
g_{\left[\begin{smallmatrix}0\\1/mn\end{smallmatrix}\right]}(\tau)^2
g_{\left[\begin{smallmatrix}1/m\\0\end{smallmatrix}\right]}(\tau)^2}
{-g_{\left[\begin{smallmatrix}1/m\\1/m\end{smallmatrix}\right]}(\tau)
g_{\left[\begin{smallmatrix}-1/m\\1/m\end{smallmatrix}\right]}(\tau)\eta(\tau)^4/
g_{\left[\begin{smallmatrix}0\\1/m\end{smallmatrix}\right]}(\tau)^2
g_{\left[\begin{smallmatrix}1/m\\0\end{smallmatrix}\right]}(\tau)^2}\quad\textrm{by Lemma \ref{wtog}}\\
&=&\frac{g_{\left[\begin{smallmatrix}1/m\\1/mn\end{smallmatrix}\right]}(\tau)
g_{\left[\begin{smallmatrix}-1/m\\1/mn\end{smallmatrix}\right]}(\tau)
g_{\left[\begin{smallmatrix}0\\1/m\end{smallmatrix}\right]}(\tau)^2}
{g_{\left[\begin{smallmatrix}1/m\\1/m\end{smallmatrix}\right]}(\tau)
g_{\left[\begin{smallmatrix}-1/m\\1/m\end{smallmatrix}\right]}(\tau)
g_{\left[\begin{smallmatrix}0\\1/mn\end{smallmatrix}\right]}(\tau)^2}.
\end{eqnarray*}
This yields by Lemma \ref{Siegellemma}(i) that $h_{m,n}(\tau)$ is a modular unit.
Moreover, if $m$ has at least two prime factors, then each of
\begin{equation*}
\left[\begin{matrix}1/m\\1/mn\end{matrix}\right],
\left[\begin{matrix}-1/m\\1/mn\end{matrix}\right],
\left[\begin{matrix}0\\1/m\end{matrix}\right],
\left[\begin{matrix}1/m\\1/m\end{matrix}\right],
\left[\begin{matrix}-1/m\\1/m\end{matrix}\right],
\left[\begin{matrix}0\\1/mn\end{matrix}\right]
\end{equation*}
has the primitive denominator that has at least two prime factors.
Therefore  $h_{m,n}(\tau)$ is a modular unit over $\mathbb{Z}$ by Lemma \ref{Siegellemma}(iii).
\end{proof}

\begin{remark}
The modular unit $h_{m,n}(\tau)$ is called a \textit{Weierstrass unit} \cite[Chapter 2, $\S$6]{K-L}.
\end{remark}

\section {Shimura's reciprocity law}

Throughout this section let $K$ be an imaginary quadratic field of discriminant $d_K$ other than $\mathbb{Q}(\sqrt{-1})$ and $\mathbb{Q}(\sqrt{-3})$.
And, set
\begin{equation}\label{theta}
  \theta_K=\frac{d_K+\sqrt{d_K}}{2}.
\end{equation}
It then belongs to $\mathbb{H}$ and forms a (relative) power integral basis for $K/\mathbb{Q}$.
\par
For a nonzero ideal $\mathfrak{f}$ of $\mathcal{O}_K$ we denote the ray class field modulo
$\mathfrak{f}$ by $K_\mathfrak{f}$. Furthermore,
if $\mathcal{O}=[N\theta_K,1]$ is the order of conductor $N$ ($\geq1$) of $K$, then we denote
the ring class field of the order $\mathcal{O}$ by $H_\mathcal{O}$.
As a consequence of the main theorem of complex multiplication we
have the following lemma.

\begin{lemma}\label{cm}
Let $N$ be a positive integer.
\begin{itemize}
\item[\textup{(i)}]
We have
$K_{(N)}=K(f(\theta_K)~|~f\in\mathcal{F}_N~\textrm{is finite at}~\theta_K)$.
\item[\textup{(ii)}]
If $N\geq2$, then
$K_{(N)}=K_{(1)}(f_{\left[\begin{smallmatrix}0\\1/N\end{smallmatrix}\right]}(\theta_K))$.
\end{itemize}
\end{lemma}
\begin{proof}
\begin{itemize}
\item[(i)] See \cite[Chapter 10, Corollary to Theorem 2]{Lang}.
\item[(ii)] See \cite[Chapter 10, Corollary to Theorem 7]{Lang}.
\end{itemize}
\end{proof}

\begin{lemma}\label{jintegral}
If $\theta\in\mathbb{H}$ is imaginary quadratic, then $j(\theta)$ is an algebraic integer.
\end{lemma}
\begin{proof}
See \cite[Theorem 4.14]{Shimura}.
\end{proof}

\begin{proposition}\label{generate}
Let $m$ \textup{(}$\geq2$\textup{)} and $n$ be positive integers.
Then
$h_{m,n}(\theta_K)$ generates
$K_{(mn)}$ over $K_{(m)}$. Moreover, if $m$ has at least two prime factors, then
$h_{m,n}(\theta_K)$ is a unit of $\mathcal{O}_{K_{(mn)}}$.
\end{proposition}
\begin{proof}
We first derive that
\begin{eqnarray*}
K_{(mn)}&=&K_{(1)}(f_{\left[\begin{smallmatrix}0\\1/mn\end{smallmatrix}\right]}(\theta_K))
\quad\textrm{by Lemma \ref{cm}(ii)}\\
&=&K_{(m)}\bigg(\frac{f_{\left[\begin{smallmatrix}0\\1/mn\end{smallmatrix}\right]}(\theta_K)
  -f_{\left[\begin{smallmatrix}1/m\\0\end{smallmatrix}\right]}(\theta_K)}
  {f_{\left[\begin{smallmatrix}0\\1/m\end{smallmatrix}\right]}(\theta_K)
  -f_{\left[\begin{smallmatrix}1/m\\0\end{smallmatrix}\right]}(\theta_K)}\bigg)\\
  &&\textrm{because}~f_{\left[\begin{smallmatrix}1/m\\0\end{smallmatrix}\right]}(\theta_K),
  f_{\left[\begin{smallmatrix}0\\1/m\end{smallmatrix}\right]}(\theta_K)\in K_{(m)}
  ~\textrm{by Lemma \ref{cm}(i)}\\
  &=&K_{(m)}(h_{m,n}(\theta_K))\quad\textrm{by (\ref{hf})}.
\end{eqnarray*}
And, if $m$ has at least two prime factors, then $h_{m,n}(\tau)$ is a modular unit
over $\mathbb{Z}$ by Proposition \ref{h}; hence $h_{m,n}(\tau)$ and $h_{m,n}(\tau)^{-1}$ are both integral over $\mathbb{Z}[j(\tau)]$. Therefore we conclude by Lemma \ref{jintegral} that $h_{m,n}(\theta_K)$ is a unit as an algebraic integer.
\end{proof}

\begin{lemma}[Shimura's reciprocity law]\label{Shimura}
Let $N$ be a positive integer and let $\mathcal{O}$ be the
order of conductor $N$ of $K$. Consider the matrix group
\begin{equation*}
W_{K,N}=\bigg\{\left[\begin{matrix}
t-B_Ks & -C_Ks\\s&t
\end{matrix}\right]\in\mathrm{GL}_2(\mathbb{Z}/N\mathbb{Z})~
|~t,s\in\mathbb{Z}/N\mathbb{Z}\bigg\},
\end{equation*}
where
\begin{equation*}
\min(\theta_K,\mathbb{Q})=X^2+B_KX+C_K=
X^2-d_KX+\frac{d_K^2-d_K}{4}.
\end{equation*}
\begin{itemize}
\item[\textup{(i)}]
We have the isomorphism
\begin{eqnarray*}
W_{K,N}/\{\pm I_2\}&\longrightarrow&\mathrm{Gal}(K_{(N)}/K_{(1)})\\
\alpha&\mapsto&(f(\theta_K)\mapsto f^\alpha(\theta_K)~|~f(\tau)\in\mathcal{F}_N~
\textrm{is finite at}~\theta_K).
\end{eqnarray*}
\item[\textup{(ii)}]
The above induces the isomorphism
\begin{equation*}
\{tI_2\in\mathrm{GL}_2(\mathbb{Z}/N\mathbb{Z})~|~t\in(\mathbb{Z}/N\mathbb{Z})^*\}/\{\pm I_2\}\longrightarrow\mathrm{Gal}(K_{(N)}/H_\mathcal{O}).
\end{equation*}
\item[\textup{(iii)}] If $M$ is a divisor of $N$, then we get the isomorphism
\begin{equation*}
\{tI_2\in\mathrm{GL}_2(\mathbb{Z}/N\mathbb{Z})~|~t\in(\mathbb{Z}/N\mathbb{Z})^*
~\textrm{with}~t\equiv\pm1\pmod{M}\}/\{\pm I_2\}\longrightarrow\mathrm{Gal}(K_{(N)}/K_{(M)}H_\mathcal{O}).
\end{equation*}
\end{itemize}
\end{lemma}
\begin{proof}
\begin{itemize}
\item[(i)] See \cite[$\S$3]{Stevenhagen}.
\item[(ii)] See \cite[Proposition 5.3]{K-S}.
\item[(iii)] This is a direct consequence of (i) and (ii).
\end{itemize}
\end{proof}

\begin{lemma}\label{invariant}
Let $N$ \textup{(}$\geq2$\textup{)} be a positive integer
for which $(N)=N\mathcal{O}_K$ is not a power of a prime ideal.
\begin{itemize}
\item[\textup{(i)}] $g_{\left[\begin{smallmatrix}0\\1/N
\end{smallmatrix}\right]}(\theta_K)^{12N}$ is a unit of $\mathcal{O}_{K_{(N)}}$.
\item[\textup{(ii)}] If $u$ is an integer prime to $N$, then
$g_{\left[\begin{smallmatrix}0\\u/N
\end{smallmatrix}\right]}(\theta_K)^{12N}$ is also a unit of $\mathcal{O}_{K_{(N)}}$.
\end{itemize}
\end{lemma}
\begin{proof}
\begin{itemize}
\item[(i)] See \cite[Remark 4.3]{J-K-S} and \cite{Ramachandra} (or, \cite[p.293]{Lang}).
\item[(ii)]
We obtain
\begin{eqnarray*}
  g_{\left[\begin{smallmatrix}0\\u/N
\end{smallmatrix}\right]}(\theta_K)^{12N}&=&
g_{{}^t(uI_2)\left[\begin{smallmatrix}0\\1/N
\end{smallmatrix}\right]}(\theta_K)^{12N}\\
&=&(g_{\left[\begin{smallmatrix}0\\1/N
\end{smallmatrix}\right]}(\tau)^{12N})^{uI_2}(\theta_K)\quad
\textrm{by Lemma \ref{Siegellemma}(i) and (ii)}\\
&=&(g_{\left[\begin{smallmatrix}0\\1/N
\end{smallmatrix}\right]}(\theta_K)^{12N})^{uI_2}\quad
\textrm{by Lemmas \ref{cm}(i) and \ref{Shimura}(i)}.
\end{eqnarray*}
Now, the result follows from (i).
\end{itemize}
\end{proof}

\begin{remark}
Here the singular value $g_{\left[\begin{smallmatrix}0\\1/N
\end{smallmatrix}\right]}(\theta_K)^{12N}$
is called a \textit{Siegel-Ramachandra invariant} modulo $(N)$, and it forms a normal basis for $K_{(N)}/K$ \cite{J-K-S}.
\end{remark}

\section {Construction of relative power integral bases}

We are ready to prove our main theorem concerning relative power integral bases.

\begin{theorem}\label{main}
Let $K$ be an imaginary quadratic field other than $\mathbb{Q}(\sqrt{-1})$ and
$\mathbb{Q}(\sqrt{-3})$.
Assume that $m$ \textup{(}$\geq2$\textup{)} and $n$ are positive integers such that
\begin{itemize}
\item[\textup{(i)}] $m$ has at least two prime factors,
\item[\textup{(ii)}] each prime factor of $mn$ splits in $K/\mathbb{Q}$.
\end{itemize}
If $L=K_{(mn)}$ and
$F=K_{(m)}H_\mathcal{O}$ with $\mathcal{O}$
the order of conductor $mn$ of $K$, then
 $h_{m,n}(\theta_K)$ forms a relative power integral basis for $L/F$.
\end{theorem}
\begin{proof}
Let $\alpha=h_{m,n}(\theta_K)$.
Since $\alpha$ is a unit of $\mathcal{O}_L$ by Proposition \ref{generate}, we have the inclusion
\begin{equation*}
\mathcal{O}_L\supseteq\mathcal{O}_F[\alpha].
\end{equation*}
\par
As for the converse, let $\beta$ be an element of $\mathcal{O}_L$. Since $L=F(\alpha)$ by
Proposition \ref{generate}, we can express $\beta$ as
\begin{equation}\label{beta}
\beta=c_0+c_1\alpha+\cdots+c_{\ell-1}\alpha^{\ell-1}\quad\textrm{for some}~c_0,c_1,\ldots,c_{\ell-1}\in F,
\end{equation}
where $\ell=[L:F]$. In order to prove the converse inclusion $\mathcal{O}_L\subseteq\mathcal{O}_F[\alpha]$ it
suffices to show that $c_0,c_1,\ldots, c_{\ell-1}\in\mathcal{O}_F$.
Multiplying
both sides of (\ref{beta}) by $\alpha^k$ ($k=0,1,\ldots,\ell-1$)
 yields
\begin{equation*}
c_0\alpha^k+c_1\alpha^{k+1}+\cdots
+c_{\ell-1}\alpha^{k+\ell-1}=\beta\alpha^k.
\end{equation*}
Now, we take the trace $\mathrm{Tr}=\mathrm{Tr}_{L/F}$ to obtain
\begin{equation*}
c_0\mathrm{Tr}(\alpha^k)+c_1\mathrm{Tr}(\alpha^{k+1})
+\cdots
+c_{\ell-1}\mathrm{Tr}(\alpha^{k+\ell-1})=\mathrm{Tr}(\beta\alpha^k).
\end{equation*}
Then we achieve a linear system (in unknowns $c_0,c_1,c_2,\ldots, c_{\ell-1}$)
\begin{equation*}
T\left[\begin{matrix}c_0\\c_1\\
\vdots\\c_{\ell-1}
\end{matrix}\right]=
\left[\begin{matrix}\mathrm{Tr}(\beta)\\
\mathrm{Tr}(\beta\alpha)\\
\vdots\\
\mathrm{Tr}(\beta\alpha^{\ell-1})\end{matrix}\right],~\textrm{where}~T=
\left[\begin{matrix}
\mathrm{Tr}(1) & \mathrm{Tr}(\alpha) & \cdots & \mathrm{Tr}(\alpha^{\ell-1})\\
\mathrm{Tr}(\alpha) & \mathrm{Tr}(\alpha^2) & \cdots & \mathrm{Tr}(\alpha^\ell)\\
\vdots & \vdots &\ddots& \vdots\\
\mathrm{Tr}(\alpha^{\ell-1}) & \mathrm{Tr}(\alpha^{\ell}) & \cdots
& \mathrm{Tr}(\alpha^{2\ell-2})
\end{matrix}\right].
\end{equation*}
Since $\alpha,\beta\in\mathcal{O}_L$, all the entries of $T$ and $\left[\begin{matrix}\mathrm{Tr}(\beta)\\
\mathrm{Tr}(\beta\alpha)\\
\vdots\\
\mathrm{Tr}(\beta\alpha^{\ell-1})\end{matrix}\right]$ lie in $\mathcal{O}_F$.
Hence we get
\begin{equation*}
c_0,c_1,\ldots, c_{\ell-1}\in\frac{1}{\det(T)}\mathcal{O}_F.
\end{equation*}
Let $\alpha_1,\alpha_2,\ldots,\alpha_\ell$ be the conjugates of $\alpha$ via $\mathrm{Gal}(L/F)$.
We then derive that
\begin{eqnarray}
\det(T)&=&\left|\begin{matrix}\sum_{k=1}^\ell\alpha_k^0
& \sum_{k=1}^\ell\alpha_k^1 & \cdots & \sum_{k=1}^\ell\alpha_k^{\ell-1}\\
\sum_{k=1}^\ell\alpha_k^1
& \sum_{k=1}^\ell\alpha_k^2 & \cdots & \sum_{k=1}^\ell\alpha_k^{\ell}\\
\vdots
& \vdots & \ddots & \vdots\\
\sum_{k=1}^\ell\alpha_k^{\ell-1}
& \sum_{k=1}^\ell\alpha_k^\ell & \cdots & \sum_{k=1}^\ell\alpha_k^{2\ell-2}\\
\end{matrix}\right|\nonumber\\
&=&\left|\begin{matrix}
\alpha_1^0 & \alpha_2^0 & \cdots & \alpha_\ell^0\\
\alpha_1^1 & \alpha_2^1 & \cdots & \alpha_\ell^1\\
\vdots & \vdots & \ddots & \vdots\\
\alpha_1^{\ell-1} & \alpha_2^{\ell-1} & \cdots & \alpha_\ell^{\ell-1}
\end{matrix}\right|
\cdot
\left|\begin{matrix}
\alpha_1^0 & \alpha_1^1 & \cdots & \alpha_1^{\ell-1}\\
\alpha_2^0 & \alpha_2^1 & \cdots & \alpha_2^{\ell-1}\\
\vdots& \vdots & \ddots & \vdots\\
\alpha_\ell^0 & \alpha_\ell^1 & \cdots & \alpha_\ell ^{\ell-1}
\end{matrix}\right|\nonumber\\
&=&\prod_{1\leq k_1<k_2\leq \ell}(\alpha_{k_1}-\alpha_{k_2})^2\quad\textrm{by the Van der Monde determinant formula}\nonumber\\
&=&\pm\prod_{\sigma_1\neq\sigma_2\in\mathrm{Gal}(L/F)}(\alpha^{\sigma_1}-\alpha^{\sigma_2})\nonumber\\
&=&\pm\prod_{\sigma_1\neq\sigma_2\in\mathrm{Gal}(L/F)}
(\alpha^{\sigma_1\sigma_2^{-1}}-\alpha)^{\sigma_2}.\label{det(T)}
\end{eqnarray}
If $\sigma$ is a nonidentity element of $\mathrm{Gal}(L/F)$, then
by Lemma \ref{Shimura}(iii) one can set
$\sigma=tI_2$ for some $t\in\mathbb{N}$ such that
\begin{equation*}
\gcd(t,mn)=1,~
t\equiv\pm1\pmod{m}~\textrm{and}~   t\not\equiv\pm1\pmod{mn}.
\end{equation*}
Thus we deduce that
\begin{eqnarray*}
\alpha^\sigma-\alpha&=&h_{m,n}(\theta_K)^\sigma-h_{m,n}(\theta_K)\\
&=&\bigg(\frac{f_{\left[\begin{smallmatrix}0\\1/mn\end{smallmatrix}\right]}(\theta_K)
  -f_{\left[\begin{smallmatrix}1/m\\0\end{smallmatrix}\right]}(\theta_K)}
  {f_{\left[\begin{smallmatrix}0\\1/m\end{smallmatrix}\right]}(\theta_K)
  -f_{\left[\begin{smallmatrix}1/m\\0\end{smallmatrix}\right]}(\theta_K)}\bigg)^\sigma
  -\frac{f_{\left[\begin{smallmatrix}0\\1/mn\end{smallmatrix}\right]}(\theta_K)
  -f_{\left[\begin{smallmatrix}1/m\\0\end{smallmatrix}\right]}(\theta_K)}
  {f_{\left[\begin{smallmatrix}0\\1/m\end{smallmatrix}\right]}(\theta_K)
  -f_{\left[\begin{smallmatrix}1/m\\0\end{smallmatrix}\right]}(\theta_K)}
  \quad\textrm{by (\ref{hf})}\\
  &=&
  \frac{f_{{}^t\sigma\left[\begin{smallmatrix}0\\1/mn\end{smallmatrix}\right]}(\theta_K)
  -f_{\left[\begin{smallmatrix}1/m\\0\end{smallmatrix}\right]}(\theta_K)}
  {f_{\left[\begin{smallmatrix}0\\1/m\end{smallmatrix}\right]}(\theta_K)
  -f_{\left[\begin{smallmatrix}1/m\\0\end{smallmatrix}\right]}(\theta_K)}
  -\frac{f_{\left[\begin{smallmatrix}0\\1/mn\end{smallmatrix}\right]}(\theta_K)
  -f_{\left[\begin{smallmatrix}1/m\\0\end{smallmatrix}\right]}(\theta_K)}
  {f_{\left[\begin{smallmatrix}0\\1/m\end{smallmatrix}\right]}(\theta_K)
  -f_{\left[\begin{smallmatrix}1/m\\0\end{smallmatrix}\right]}(\theta_K)}
  \quad\textrm{by Lemmas \ref{Shimura}(iii) and \ref{Fricketransf}}\\
  &=&\frac{f_{\left[\begin{smallmatrix}0\\t/mn\end{smallmatrix}\right]}(\theta_K)
  -f_{\left[\begin{smallmatrix}0\\1/mn\end{smallmatrix}\right]}(\theta_K)}
  {f_{\left[\begin{smallmatrix}0\\1/m\end{smallmatrix}\right]}(\theta_K)
  -f_{\left[\begin{smallmatrix}1/m\\0\end{smallmatrix}\right]}(\theta_K)}\\
  &=&\frac{\wp_{\left[\begin{smallmatrix}0\\t/mn\end{smallmatrix}\right]}(\theta_K)
  -\wp_{\left[\begin{smallmatrix}0\\1/mn\end{smallmatrix}\right]}(\theta_K)}
  {\wp_{\left[\begin{smallmatrix}0\\1/m\end{smallmatrix}\right]}(\theta_K)
  -\wp_{\left[\begin{smallmatrix}1/m\\0\end{smallmatrix}\right]}(\theta_K)}
  \quad\textrm{by the definition (\ref{Fricke})}\\
  &=&\frac{g_{\left[\begin{smallmatrix}0\\(t+1)/mn\end{smallmatrix}\right]}(\theta_K)
  g_{\left[\begin{smallmatrix}0\\(t-1)/mn\end{smallmatrix}\right]}(\theta_K)
  g_{\left[\begin{smallmatrix}0\\1/m\end{smallmatrix}\right]}(\theta_K)^2
  g_{\left[\begin{smallmatrix}1/m\\0\end{smallmatrix}\right]}(\theta_K)^2}
  {g_{\left[\begin{smallmatrix}1/m\\1/m\end{smallmatrix}\right]}(\theta_K)
  g_{\left[\begin{smallmatrix}-1/m\\1/m\end{smallmatrix}\right]}(\theta_K)
  g_{\left[\begin{smallmatrix}0\\t/mn\end{smallmatrix}\right]}(\theta_K)^2
  g_{\left[\begin{smallmatrix}0\\1/mn\end{smallmatrix}\right]}(\theta_K)^2}
  \quad\textrm{by Lemma \ref{wtog}}.
\end{eqnarray*}
Since each of
\begin{equation*}
\left[\begin{matrix}
0\\1/m
\end{matrix}\right],
\left[\begin{matrix}1/m\\0\end{matrix}\right],
\left[\begin{matrix}1/m\\1/m\end{matrix}\right],
\left[\begin{matrix}-1/m\\1/m\end{matrix}\right],
\left[\begin{matrix}0\\t/mn\end{matrix}\right],
\left[\begin{matrix}0\\1/mn\end{matrix}\right]
\end{equation*}
has the primitive denominator with at least two prime factors by the hypothesis (i),
the values
\begin{equation*}
g_{\left[\begin{smallmatrix}
0\\1/m
\end{smallmatrix}\right]}(\theta_K),
g_{\left[\begin{smallmatrix}1/m\\0\end{smallmatrix}\right]}(\theta_K),
g_{\left[\begin{smallmatrix}1/m\\1/m\end{smallmatrix}\right]}(\theta_K),
g_{\left[\begin{smallmatrix}-1/m\\1/m\end{smallmatrix}\right]}(\theta_K),
g_{\left[\begin{smallmatrix}0\\t/mn\end{smallmatrix}\right]}(\theta_K),
g_{\left[\begin{smallmatrix}0\\1/mn\end{smallmatrix}\right]}(\theta_K)
\end{equation*}
are units as algebraic integers by Lemmas \ref{Siegellemma}(iii) and \ref{jintegral}.
On the other hand, let us put
\begin{equation*}
\frac{t+1}{mn}=\frac{a}{N}\quad\textrm{for some relatively prime positive integers $N$ and $a$}.
\end{equation*}
Since $t\not\equiv\pm1\pmod{mn}$, we get $N\geq2$.
Moreover, $(N)=N\mathcal{O}_K$ is not a power of a prime ideal by the hypothesis (ii).
So $g_{\left[\begin{smallmatrix}0\\(t+1)/mn\end{smallmatrix}\right]}(\theta_K)
=g_{\left[\begin{smallmatrix}0\\a/N\end{smallmatrix}\right]}(\theta_K)$
is a unit as an algebraic integer by Lemma \ref{invariant}(ii).
In a similar fashion, we also see that  $g_{\left[\begin{smallmatrix}0\\(t-1)/mn\end{smallmatrix}\right]}(\theta_K)$ is a unit as an algebraic integer. Therefore $\alpha^\sigma-\alpha$ is a unit of $\mathcal{O}_L$.
This implies that
$\det(T)$ is a unit of $\mathcal{O}_F$ by (\ref{det(T)}), and hence we get the converse inclusion
\begin{equation*}
  \mathcal{O}_L\subseteq\mathcal{O}_F[\alpha],
\end{equation*}
as desired. This completes the proof.
\end{proof}

\bibliographystyle{amsplain}

\address{
Pohang Mathematics Institute\\
POSTECH \\
Pohang 790-784\\
Republic of Korea} {hoyunjung@postech.ac.kr}
\address{
Department of Mathematical Sciences \\
KAIST \\
Daejeon 305-701 \\
Republic of Korea} {jkkoo@math.kaist.ac.kr}
\address{
Department of Mathematics\\
Hankuk University of Foreign Studies\\
Yongin-si, Gyeonggi-do 449-791\\
Republic of Korea } {dhshin@hufs.ac.kr}

\end{document}